\documentclass[12pt]{article}
\usepackage{amssymb,amsfonts,amsmath,amsthm,amsrefs,dsfont,hyperref}
\usepackage[letterpaper, portrait, margin=28mm]{geometry}

\newcommand{\Z}{\mathbb{Z}}

\newcommand{\R}{\mathbb{R}}

\newcommand{\N}{\mathbb{N}}

\newcommand{\Bo}{\mathrm{B}}
\newcommand{\So}{\mathrm{S}}
\newcommand{\Bbo}{\overline{\mathrm{B}}}
\newcommand{\8}{\infty}

\newcommand{\spa}{\mathrm{span}}
\newcommand{\Ker}{\mathrm{Ker~}}

\newcommand{\Lo}{\mathcal{L}}

\newcounter{dummy} \numberwithin{dummy}{section}
\newtheorem{theorem}[dummy]{Theorem}
\newtheorem{lemma}[dummy]{Lemma}

\newtheorem{proposition}[dummy]{Proposition}
\newtheorem{corollary}[dummy]{Corollary}
\newtheorem{question}[dummy]{Question}
\theoremstyle{remark}
\newtheorem{remark}[dummy]{Remark}
\newtheorem{example}[dummy]{Example}

\begin{document}

\title{Disjointly non-singular operators on order continuous Banach lattices complement the unbounded norm topology}
\author{Eugene Bilokopytov\footnote{Email address bilokopy@ualberta.ca, erz888@gmail.com.}}
\maketitle

\begin{abstract}
In this article we investigate the disjointly non-singular (DNS) operators. Following \cite{gmm} we say that an operator $T$ from a Banach lattice $F$ into a Banach space $E$ is DNS, if no restriction of $T$ to a subspace generated by a disjoint sequence is strictly singular. We partially answer a question from \cite{gmm} by showing that this class of operators forms an open subset of $\Lo\left(F,E\right)$ as soon as $F$ is order continuous. Moreover, we show that in this case $T$ is DNS if and only if the norm topology is the minimal topology which is simultaneously stronger than the unbounded norm topology and the topology generated by $T$ as a map (we say that $T$ ``complements'' the unbounded norm topology in $F$). Since the class of DNS operators plays a similar role in the category of Banach lattices as the upper semi-Fredholm operators play in the category of Banach spaces, we investigate and indeed uncover a similar characterization of the latter class of operators, but this time they have to complement the weak topology.\medskip


MSC2020 46B42, 47A53, 47B60.
\end{abstract}

\section{Introduction}

The class of disjointly non-singular (DNS) operators was introduced in \cite{gmm} to single out one of the properties of Tauberian operators on $L_{1}$, and simultaneously as the class of operators that demonstrates an opposite behavior to disjointly strictly singular (DSS) operators. Namely, an operator $T$ from a Banach lattice $F$ into a Banach space $E$ is DNS, if no restriction of $T$ to a subspace generated by a disjoint sequence is strictly singular. It was shown in \cite{gmm} that this condition is equivalent to the fact that $T$ does not send normalized disjoint sequences into null sequences.
In the aforementioned paper a thorough analysis of these operators on $L_{p}$ spaces was performed. The authors have showed that $T\in \Lo\left(L_{p},E\right)$ is DNS if and only if it is ``uniformly DNS'', i.e. there is $\delta>0$ such that no normalized disjoint sequence in $L_{p}$ is mapped into $\delta\Bo_{E}$. It is an easy consequence of the last result that the collection of DNS operators forms an open subset of $\Lo\left(L_{p},E\right)$. A related notion that was introduced was the class of the dispersed subspaces of Banach lattices, which in the case of $L_{p}$ coincide with the strongly embedded subspaces, i.e. subspaces on which the topology of convergence in measure coincides with the norm topology.\medskip

In this paper we abstract some of the results of \cite{gmm} from the $L_{p}$ spaces to the class of order-continuous Banach lattices. We make use of the fact that convergence in measure on $L_{p}$ is a specific case of the unbounded norm (un) topology on Banach lattices, studied e.g. in \cite{kmt}. It turns out that the dispersed subspaces of an order continuous Banach lattice are precisely the subspaces on which the un-topology coincides with the norm topology. Similarly, one of the characterizations of DNS operators on $L_{p}$ spaces in terms of convergence in measure allows a generalization for the case of the order-continuous Banach lattices in terms of the un-topology. Namely, $T$ is DNS if and only if the norm topology is the minimal topology which is simultaneously stronger than the unbounded norm topology and the topology generated by $T$ as a map (we say that $T$ ``complements'' the un-topology in $F$). The ``uniform'' characterization mentioned in the first paragraph is still valid in this more general context, and so the set of DNS operators forms an open subset of $\Lo\left(F,E\right)$, which partially answers one of the questions posed in \cite{gmm}.\medskip

The observations above motivated us to perform a more general study of operators complementary to a given topology in the sense as above, as well as subspaces on which this topology coincides with the norm topology. We dedicate Section \ref{comple} to this general framework. We remark here that the considered construction is somewhat reminiscent of that in \cite{ad}.

Having developed the machinery of the operators that complement a topology inspires an application to the only topology that is present on a Banach space by default -- the weak topology. It turns out that the operators complementary to the weak topology are precisely the upper semi-Fredholm (USF) operators, i.e. the operators which have closed ranges and finitely dimensional kernels. In Section \ref{usfe} we collect several known and new equivalent characterizations of the USF operators, which includes ones remarkably similar to the characterizations of the DNS operators in a more specific context (finitely dimensional subspaces, weak topology and basic sequences correspond to the dispersed subspaces, un-topology and disjoint sequences). Moreover, reflexive spaces are precisely the spaces on which for every USF operator there is $\delta>0$ such that no normalized basic sequence is mapped into the $\delta$-ball (in general, an operator is USF if and only if it does not send normalized basic sequences into null sequences). It is safe to say that the DNS operators play a similar role in the category of Banach lattices as the USF operators play in the category of Banach spaces. In fact, this observation is corroborated by the fact, also established in \cite{gmm}, that on the Banach lattices whose lattice structure is the simplest (i.e. order continuous discrete) the notions of DNS and USF operators coincide.\medskip

In Section \ref{dispe} we focus on the dispersed subspaces, and consider several variations of the concept which coincide if $F$ is order continuous. Section \ref{dnsdss} is mostly devoted to characterizations of DNS operators mentioned above. We remark here that these characterizations mirror the variations of dispersed subspaces and coincide when  $F$ is order continuous. The last Section \ref{lns} briefly discusses a possible generalization of DNS operators -- the LNS operators whose definition differs from DNS by only considering positive disjoint sequences. This class is meant to be a counterpart for the lattice strictly singular (LSS) operators, studied e.g. in \cite{flt}.

\section{Complementary topologies}\label{comple}

In this section $E$ is a normed space and $\tau$ is a linear topology on $E$ which is weaker than the norm topology. Throughout the article $\Bo_{E}$ will stand for the open unit ball of $E$, $\Bbo_{E}$ -- for the closed unit ball, and $S_{E}=\partial\Bo_{E}$ -- for the unit sphere.

Recall that a Hausdorff topological space is called \emph{compactly generated}, or a \emph{k-space} whenever each set which has closed intersections with all compacts is closed itself. It is easy to see that all metrizable and all locally compact Hausdorff spaces are compactly generated. A topological space is \emph{Frechet-Urysohn} if the closure of any set coincides with its sequential closure. This condition is equivalent to the fact that every convergent net contains a sequence convergent to the same limit. Frechet-Urysohn spaces are compactly generated. Additional information about compactly generated spaces see in \cite[3.3]{engelking}, and about Frechet-Urysohn spaces -- in \cite[Chapter 14]{kkl}. We will need the following lemma:

\begin{lemma}\label{fu}
If $F$ is a Frechet-Urysohn topological vector space and $G$ is a semi-normed space, then $F\times G$ is Frechet-Urysohn.
\end{lemma}
\begin{proof}
Let $\left\{\left(f_{i},g_{i}\right)\right\}_{i\in I}$ be a net such that $f_{i}\xrightarrow[i\in I]{}0_{F}$ and $g_{i}\xrightarrow[i\in I]{}0_{G}$. For every $n\in\N$ there is $i_{n}$ such that $\|g_{i}\|\le\frac{1}{n}$, for every $i\ge i_{n}$, and so there is $\left\{i_{mn}\right\}_{m\in\N}$ such that $i_{mn}\ge i_{n}$, for every $m\in\N$, and $f_{i_{mn}}\xrightarrow[m\to\8]{} 0_{F}$. Therefore, there are increasing sequences $\left\{m_{k}\right\}_{k\in\N}$ and $\left\{n_{k}\right\}_{k\in\N}$ such that $f_{i_{k}}\xrightarrow[k\to\8]{} 0_{F}$, where $i_{k}=i_{m_{k},n_{k}}$, for $k\in\N$ (see \cite[Appendix 1]{yam}; note that Frechet-Urysohn property is called ``sequential'' there). At the same time, since $i_{k}\ge i_{n_{k}}$, it follows that $\|g_{i_{k}}\|\le \frac{1}{n_{k}}$, for every $k\in\N$, and so $g_{i_{k}}\xrightarrow[k\to\8]{} 0_{G}$. Thus, $\left(f_{i_{k}},g_{i_{k}}\right)\xrightarrow[k\to\8]{} \left(0_{F},0_{G}\right)$.
\end{proof}

We say that topologies $\pi_{1}$ and $\pi_{2}$ on $E$ \emph{generate} the topology $\pi$, if $\pi$ is the minimal topology on $E$ which is stronger than both $\pi_{1}$ and $\pi_{2}$ (we denote it $\pi=\pi_{1}\vee\pi_{2}$). It is easy to see that $\pi$ is precisely the topology of the diagonal of $\left(E,\pi_{1}\right)\times \left(E,\pi_{2}\right)$, and in particular a set is compact (a net is convergent) in $\pi$ if and only if it is compact (convergent) in both $\pi_{1}$ and $\pi_{2}$. Since a subset of a Frechet-Urysohn topological space is Frechet-Urysohn, it follows from the preceding lemma that if $\pi_{1}$ is Frechet-Urysohn and $\pi_{2}$ is generated be  a semi-norm, then $\pi$ is Frechet-Urysohn.

\begin{theorem}\label{compl}For a linear topology $\pi$ on $E$, which is weaker than the norm topology, the following conditions are equivalent:
\item[(i)] There is no net in $\So_{E}$, which is null with respect to both $\tau$ and $\pi$;
\item[(ii)] There are open neighborhoods $U\in\tau$ and $V\in\pi$ of $0_{E}$ such that $U\cap V\cap \So_{E}=\varnothing$;
\item[(iii)] $\tau$ and $\pi$ generate the norm topology on $E$ (or on $\Bo_{E}$, or on $\Bbo_{E}$);
\item[(iv)] The topology $\tau\vee\pi$ is compactly generated and $K$ is norm-compact if and only if it is compact with respect to both $\tau$ and $\pi$, for every $K\subset E$ (or $K\subset\Bo_{E}$, or $K\subset\Bbo_{E}$).\medskip

If additionally $\pi$ is generated by a semi-norm $\rho$, then the conditions above are equivalent to
\item[(v)] There is $\delta>0$ such that $0_{E}$ is $\tau$-separated from $\left\{f\in \So_{E},~ \rho\left(f\right)\le\delta\right\}$;
\item[(vi)] There is $\delta>0$ such that the inclusion of the set $\left\{e\in E,~ \rho\left(e\right)\le\delta\|e\|\right\}$ into $E$ is $\tau$-to-norm continuous at $0_{E}$.\medskip

If on top of that $\tau$ is Frechet-Urysohn, then the conditions above are equivalent to
\item[(vii)] There is no sequence in $\So_{E}$, which is null with respect to both $\tau$ and $\pi$;
\item[(viii)] $K$ is norm-compact if and only if it is compact with respect to both $\tau$ and $\pi$, for every $K\subset E$ (or $K\subset\Bo_{E}$, or $K\subset\Bbo_{E}$).
\end{theorem}
\begin{proof}
(i)$\Rightarrow$(ii): Assume the contrary. Let $\mathcal{U}$ and $\mathcal{V}$ be bases at $0_{E}$ of $\tau$ and $\pi$ respectively, directed downwards. For every $U\in \mathcal{U}$ and $V\in \mathcal{V}$ there is $e_{UV}\in U\cap V\cap  \So_{E}$. Then, the net $\left\{e_{UV}\right\}_{U\in \mathcal{U},~V\in \mathcal{V}}$ is null simultaneously in $\tau$ and $\pi$.\medskip

(ii)$\Rightarrow$(iii, for $E$): Since all topologies involved are linear, it is enough to show that there are open neighborhoods $U\in\tau$ and $V\in\pi$ of $0_{E}$ such that $U\cap V\subset \Bo_{E}$. Indeed, there are open neighborhoods $U_{0}\in\tau$ and $V_{0}\in\pi$ of $0_{E}$ such that $U_{0}\cap V_{0}\cap \So_{E}=\varnothing$. Since $\tau$ and $\pi$ are linear topologies, there are balanced $U\in\tau$ and $V\in\pi$ such that $U\subset U_{0}$ and $V\subset V_{0}$ (see \cite[Theorem 4.3.6]{bn}). Then, for every $r\ge 1$ we have $0_{E}\in \frac{1}{r}\left(U\cap V\right) \subset U\cap V\subset U_{0}\cap V_{0}\subset E\backslash \So_{E}$. Hence, $U\cap V\cap r\So_{E}=\varnothing$, for every $r\ge 1$, and so $U\cap V\subset B_{E}$.\medskip

(iii, for $\Bo_{E}$)$\Rightarrow$(i): If a net in $\So_{E}\subset 2\Bo_{E}$ is null simultaneously in $\tau$ and $\pi$, then it is also norm-null, and so not contained in $\So_{E}$.\medskip

(iii)$\Leftrightarrow$(iv): From the comment before the theorem, a set is compact in $\tau\vee\pi$ if and only if it is compact in both $\tau$ and $\pi$; moreover, if $\tau\vee\pi$ is the norm topology, it is certainly compactly generated. Conversely, if $\tau\vee\pi$ is compactly generated and has the same compact sets as the norm topology, these topologies coincide, since they are both determined by their compact sets.\medskip

It is clear that (v) is just a reformulation of (ii) for the case of a semi-norm, and since $0_{E}$ is norm-separated from  $\left\{e\in \So_{E},~ \rho\left(e\right)\le\delta\right\}=\So_{E}\cap \left\{e\in E,~ \rho\left(e\right)\le\delta\|e\|\right\}$, we have (vii)$\Rightarrow$(v); (v)$\Rightarrow$(vii) is proven similarly to (ii)$\Rightarrow$(iii).\medskip

(i)$\Rightarrow$(vii) and (iv)$\Rightarrow$(viii) are trivial. The converses follows from the observation before the theorem that if $\tau$ is Frechet-Urysohn and $\pi$ is generated by a semi-norm, then $\tau\vee\pi$ is Frechet-Urysohn, and so compactly generated.
\end{proof}

If $\pi$ is the trivial topology, we get the following.

\begin{corollary}\label{eq}
$\tau$ is equal to the norm topology if and only if $0_{E}$ is $\tau$-separated from $\So_{E}$, and inf and only if $\Bbo_{E}\not\subset\overline{\So_{E}}^{\tau}$.
\end{corollary}
\begin{proof}
We only need to show that if $0_{E}$ is not $\tau$-separated from $\So_{E}$, then $\Bbo_{E}\subset\overline{\So_{E}}^{\tau}$. Let $f\in \Bbo_{E}$. Since if $f\in \So_{E}$, then trivially $f\in \overline{\So_{E}}^{\tau}$, we may assume that $\delta=\|f\|<1$. As $0_{E}$ is not $\tau$-separated from $\So_{E}$, there is a $\tau$-null net $\left\{e_{i}\right\}_{i\in I}\subset\So_{E}$. For every $i\in I$ the ray $\left\{f+te_{i}\right\}_{t\ge 0}$ intersects $\So_{E}$ exactly once; let $f_{i}=f+t_{i}e_{i}$ be the intersection. It follows that $t_{i}-\delta =\|t_{i}e_{i}\|-\|f\|\le \|t_{i}e_{i}-f\|=\|f_{i}\|=1$, from where $0\le t_{i}\le 1+\delta$, for every $i\in I$, and so $\left\{f_{i}\right\}_{i\in I}$ is a net in $\So_{E}$ such that $f_{i}=f+t_{i}e_{i}\xrightarrow[\tau]{i\in I}f$.
\end{proof}

We will say that linear topologies $\tau$ and $\pi$ on a normed space $E$ are \emph{complementary} if they satisfy the conditions of the Theorem \ref{compl}. Note that if $\tau$ and $\pi$ are complementary, then the stronger topologies $\tau'$ and $\pi'$  are also complementary. In the case $\pi$ is generated by a (continuous) semi-norm $\rho$ we will say that $\rho$ is $\tau$-complementary.\medskip

For semi-norms $\rho$ and $\lambda$ on $E$, define $d\left(\rho,\lambda\right)=\sup\limits_{e\in\So_{E}}\left|\rho\left(e\right)-\lambda\left(e\right)\right|$, and let $\delta_{\tau}\left(\rho\right)$ be the supremum of $\delta>0$ such that $0_{E}$ is $\tau$-separated from $\left\{e\in \So_{E},~ \rho\left(e\right)\le\delta\right\}$ (if no such $\delta$ exists put $\delta_{\tau}\left(\rho\right)=0$). Clearly, if $\lambda\ge \rho$, then $\delta_{\tau}\left(\lambda\right)\ge \delta_{\tau}\left(\rho\right)$, and $\delta_{\tau}\left(\rho\right)>0$ if and only if $\rho$ is $\tau$-complementary.

\begin{proposition}\label{compl1} Let $\rho$ and $\lambda$ be continuous semi-norms on $E$. Then:
\item[(i)] If $\tau$ is equal to the norm topology, then $\delta_{\tau}\left(\rho\right)=+\8$. Otherwise, $\delta_{\tau}\left(\rho\right)<+\8$.
\item[(ii)] If $\rho$ is $\tau$-complementary, then $\tau$ coincides with the norm topology on $\Ker \rho$, and more broadly, on any set where $\rho$ is $\tau$-continuous.
\item[(iii)] If $\lambda\ge \rho$, and $\rho$ is $\tau$-complementary then so is $\lambda$. Consequently, equivalent semi-norms are (not)  $\tau$-complementary simultaneously.
\item[(iv)] If $\left|\rho-\lambda\right|$ is continuous with respect to $\tau$ on $\Bbo_{E}$ and $\rho$ is $\tau$-complementary, then so is $\lambda$.
\item[(v)] $\left|\delta_{\tau}\left(\rho\right)-\delta_{\tau}\left(\lambda\right)\right|\le d\left(\rho,\lambda\right)$. Consequently, the set of $\tau$-complementary semi-norms is open with respect to $d$.
\end{proposition}
\begin{proof}
(i): Since $\rho$ is a continuous semi-norm, it is bounded on $\So_{E}$. Hence, $\delta_{\tau}\left(\rho\right)=+\8$ is equivalent to the fact that $0_{E}$ is $\tau$-separated from $\So_{E}$. This in turn is equivalent to $\tau$ being equal to the norm topology, according to Corollary \ref{eq}.

(ii) follows from the fact that $\rho$ generates the trivial topology on $\Ker \rho$, and in general $\rho$ and $\tau$ generate $\tau$ on any set where $\rho$ is $\tau$-continuous. (iii) follows directly from the definition.

(iv): If $\left\{f_{i}\right\}_{i\in I}\subset\So_{E}$ is null with respect to both $\lambda$ and $\tau$, then it is null in $\left|\rho-\lambda\right|$, and so null in $\rho$. Contradiction.

(v): We may assume that $\delta_{\tau}\left(\rho\right)\ge\delta_{\tau}\left(\lambda\right)$. If $\delta_{\tau}\left(\rho\right)\le d\left(\rho,\lambda\right)$ we are done. Otherwise, let $d\left(\rho,\lambda\right)<\delta<\delta_{\tau}\left(\rho\right)$. Then, $0_{E}$ is $\tau$-separated from the set $$\left\{e\in \So_{E},~ \rho\left(e\right)\le\delta\right\}\supset\left\{e\in \So_{E},~ \lambda\left(e\right)\le\delta-d\left(\rho,\lambda\right)\right\},$$ and so $\delta_{\tau}\left(\lambda\right)\ge \delta-d\left(\rho,\lambda\right)$. Since $\delta$ was chosen arbitrarily, the result follows.
\end{proof}

Let us consider a special type of semi-norms on $E$. Namely, for a subspace $F$ of $E$ let $\rho_{F}$ be the distance to $F$. Before characterising subspaces $F$, for which $\rho_{F}$ is complementary to $\tau$ consider the following useful lemma.

\begin{lemma}\label{sph} If $e\in \So_{E}$ and $f\in E\backslash\left\{0_{E}\right\}$, then $\left\|e-\frac{f}{\|f\|}\right\|\le 2\|e-f\|$.
\end{lemma}
\begin{proof}
Let $\delta=\|e-f\|$. Then $\left|1- \|f\|\right|\le \delta$. Hence, $\left\|e-\frac{f}{\|f\|}\right\|\le\left\|e-f+f-\frac{f}{\|f\|}\right\|\le \|e-f\|+ \left|1- \|f\|\right|\le 2\delta$.
\end{proof}

Therefore, if $F$ is a subspace of $E$, and $\|e\|=1$, then $d\left(e,\So_{F}\right)\le 2\rho_{F}\left(e\right)$.

\begin{proposition}\label{compl2} For a subspace $F$ of $E$ the following conditions are equivalent:
\item[(i)] $0_{E}$ is $\tau$-separated from $\So_{F}$;
\item[(ii)] $\tau$ coincides with the norm topology on $F$;
\item[(iii)] $\rho_{F}$ is $\tau$-complementary on $E$.
\end{proposition}
\begin{proof}
(ii)$\Rightarrow$(i) follows from the fact that $0_{E}$ is norm-separated from $\So_{F}$, and (iii)$\Rightarrow$(ii) follows from part (i) of Proposition \ref{compl1}.

(i)$\Rightarrow$(iii): Assume that there is a net $\left\{e_{i}\right\}_{i\in I}\subset \So_{E}$, which is null with respect to both $\tau$ and $\rho_{F}$. For every $i\in I$ if $e_{i}\in F$ put $f_{i}=e_{i}$; otherwise from Lemma \ref{sph}, there is $f_{i}\in \So_{F}$ such that $\|e_{i}-f_{i}\|\le 3\rho_{F}\left(e_{i}\right)$. Since $\tau$ is weaker than the norm topology, $\left\{e_{i}-f_{i}\right\}_{i\in I}$ is $\tau$-null, from where $\left\{f_{i}\right\}_{i\in I}$ is a $\tau$-null net in $\So_{F}$. Contradiction.
\end{proof}

There are several ways to introduce a metric on the set of the closed subspaces of a normed linear space, and we will consider some of them. Namely, if $F$ and $G$ are subspaces of $E$, let $d_{1}\left(F,G\right)=\sup\limits_{f\in \So_{F}} \rho_{G}\left(f\right)\vee \sup\limits_{g\in \So_{G}} \rho_{F}\left(g\right)$, $d_{2}\left(F,G\right)=\sup\limits_{f\in \So_{F}} d\left(f,\So_{G}\right)\vee \sup\limits_{g\in \So_{G}} d\left(g,\So_{F}\right)$, and $d_{3}\left(F,G\right)=d\left(\rho_{F}, \rho_{G}\right)$. Note that $d_{2}\left(F,G\right)$ is equal to the Hausdorff distance between $\So_{F}$ and $\So_{G}$, or equivalently, between $\Bbo_{F}$ and $\Bbo_{G}$. While $d_{2}$ and $d_{3}$ are obviously metrics, this is not true for $d_{1}$, but it follows from Lemma \ref{sph} that $d_{1}\le d_{2}\le 2d_{1}$. Also, since $\rho_{F}$ and $\rho_{G}$ vanish on $F$ and $G$ respectively, we have $d_{1}\le d_{3}$. Finally, $d_{3}\le 2d_{1}$ can be deduced from \cite[Chapter IV, Lemma 2.2]{kato}. Hence, all introduced distances generate the same topology on the collection of the closed subspaces of a normed space $E$ (we will call it the \emph{gap topology}). For further  details see \cite[Chapter IV]{kato} and \cite{ostrov}. Applying part (v) of Proposition \ref{compl2} to the distance $d_{3}$ we get the following result.

\begin{corollary}\label{open}The set of the closed subspaces of $E$, on which  $\tau$ coincides with the norm topology is gap-open.
\end{corollary}

Let us now consider another class of semi-norms on  $E$. Namely, if $T:E\to F$ is a linear operator into another normed space $F$, define a seminorm $\rho_{T}$ on $E$ by $\rho_{T}\left(e\right)=\|Te\|$. Note that $\rho_{T}$ can tell us a lot about $T$. First, the topology generated by $T$ as a map into a topological space $F$ is the same as the one generated by $\rho_{T}$, and in particular $T$ is continuous whenever $\rho_{T}$ is. Furthermore, the quotient map from $E$ onto $E\slash\Ker \rho_{T}$ (with the norm induced by $\rho_{T}$) is unitarily equivalent to $T$, and so $\rho_{T}$ completely characterizes geometrical behavior of $T$. In particular, if $F$ is a closed subspace of $E$, then $\rho_{F}=\rho_{Q_{F}}$, where $Q_{F}$ is the quotient map from $E$ onto $E\slash F$. In fact the class of $\rho_{T}$ exhausts the supply of semi-norms on $E$: if $\rho$ is a seminorm, then $\rho=\rho_{Q_{\Ker \rho}}$, where $E\slash \Ker \rho$ is endowed with the norm induced by $\rho$.

It is interesting to characterize more classes of operators in terms of the semi-norms induced by them. For example, a continuous operator $T$ is bounded from below if and only if $\rho_{T}$ is equivalent to $\|\cdot\|$, and it follows from \cite[Chapter 2, Theorem 12]{groth} that compactness of $T$ is equivalent to continuity of $\rho_{T}$ on $B_{E}$ with respect to the weak topology. Let us characterize the class of operators that are essentially determined by their kernels.

\begin{proposition}\label{closed}An operator $T:E\to F$ between Banach spaces $E$ and $F$ is continuous and has a closed range if and only if $\rho_{T}$ is equivalent to $\rho_{\Ker T}$.
\end{proposition}
\begin{proof}
Let $H=\Ker T$.

Necessity: Recall that $T$ can be factored through the quotient $Q_{H}$, as $T=SQ_{H}$, where $S$ is a continuous injection. Then the range of $S$ is the same as of $T$, and so $S$ is a bijection onto a closed subspace of a Banach space. Hence, $S$ is bounded from below, and so $\rho_{T}=\rho_{SQ_{H}}\sim \rho_{Q_{H}}=\rho_{H}$.

Sufficiency: Since $\rho_{H}\le \|\cdot\|$ it follows that $\rho_{T}$ is continuous, from where $T$ is continuous, and $H$ is closed. Moreover, we still have $T=SQ_{H}$ with $\rho_{SQ_{H}}=\rho_{T}\sim \rho_{H}=\rho_{Q_{H}}$, from where $S$ is an isomorphism from a Banach space $E\slash H$ onto $TE$. Thus, $TE$ is closed.
\end{proof}

As a consequence, a restriction of an operator with a closed range to a closed subspace has a closed range.\medskip

We will say that an operator $T$ from $E$ into a normed space $F$ is $\tau$-\emph{complementary} if $\rho_{T}$ is  $\tau$-complementary. Note, that every $\tau$-complementary linear operator is continuous, and $\tau$ coincides with the norm topology on every set where $T$ is $\tau$-continuous, including $\Ker T$. The set of all $\tau$-complementary operators will be denoted by $U_{\tau}\left(E,F\right)$.

\begin{proposition}\label{complo} Let $F$ be a normed space and let $T\in \Lo\left(E,F\right)$. Then:
\item[(i)] If $S\in \Lo\left(E,F\right)$, then $d\left(\rho_{T}, \rho_{S}\right)\le \|T-S\|$. In particular, if $T\in U_{\tau}\left(E,F\right)$ and $\|T-S\|<\delta_{\tau}\left(\rho_{T}\right)$, then $S\in U_{\tau}\left(E,F\right)$. Consequently, $U_{\tau}\left(E,F\right)$ is open in $\Lo\left(E,F\right)$.
\item[(ii)] If $G$ is a normed space, and $S\in \Lo\left(F,G\right)$ is such that $ST\in U_{\tau}\left(E,G\right)$, then $T\in U_{\tau}\left(E,F\right)$.
\item[(iii)] If $G$ is a normed space, and $S\in \Lo\left(F,G\right)$ is bounded from below, then $ST$ and $T$ are (not) $\tau$-complementary simultaneously. Same for $T$ and $TR$, for a bounded from below $R\in \Lo\left(E\right)$.
\item[(iv)] If $E$ and $F$ are Banach spaces, and $T$ has a closed range, then $T\in U_{\tau}\left(E,F\right)$ if and only if $\tau$ coincides with the norm topology on $\Ker T$.
\end{proposition}
\begin{proof}
(i): the first claim follows from the triangle inequality; the second claim follows from part (iv) of Proposition \ref{compl1}.

(ii): Since $\rho_{ST}\le \|S\|\rho_{T}$, the claim follows from part (iii) of Proposition \ref{compl1}. (iii) is proven similarly.

(iv) follows from combining Proposition \ref{compl2} with Proposition \ref{closed}.
\end{proof}

\begin{remark}\label{complo2}Since $\|\left.T\right|_{H}-\left.S\right|_{H}\|\le \|T-S\|$, for any $T,S\in \Lo\left(E,F\right)$ and a subspace $H$ of $E$, it follows that $U_{\tau}\left(H,F\right)\cap \Lo\left(E,F\right)$ is open. Also, if $\pi$ is the topology generated by $\tau$ and $\rho_{T}$, then applying Corollary \ref{open} to $\pi$ we conclude that the collection of subspaces of $E$ on which $T$ is $\tau$-complementary is gap-open. \qed\end{remark}

\section{Operators that complement the weak topology}\label{usfe}

In an abstract Banach space the only natural linear topology distinct from the norm topology is the weak topology. Before considering the class of the weakly complementary operators let us recall some facts about basic sequences in Banach spaces (more details see in \cite[Section 1.3]{ak} and \cite[Section 2]{gmm}).

\begin{theorem}\label{basic}
Let $E$ be a Banach space and let $\left\{e_{n}\right\}_{n\in\N}\subset\So_{E}$ be a basic sequence, with the corresponding biorthogonal sequence $\left\{\nu_{n}\right\}_{n\in\N}\subset r\Bo_{E}$, for some $r\ge 1$. Then:\medskip
\item[(i)] If $F$ is a Banach space, and $\left\{f_{n}\right\}_{n\in\N}\subset F$ is such that $\sum\limits_{n\in\N}\|f_{n}\|=a<+\8$, then the operator $S:E\to F$ defined by $Se=\sum\limits_{n\in\N}\nu_{n}\left(e\right)f_{n}$ is compact with $\|S\|\le ar$. Moreover, $\left\|\left.S\right|_{E_{m}}\right\|\xrightarrow[m\to\8]{}0$, where $E_{m}=\overline{\spa}\left\{e_{n}\right\}_{n\ge m}$.\medskip
\item[(ii)] If $\left\{g_{n}\right\}_{n\in\N}$ is such that $\sum\limits_{n\in\N}\|g_{n}-e_{n}\|=a<\frac{1}{r}$, then the operator $T:E\to E$ defined by $Te=e+\sum\limits_{n\in\N}\nu_{n}\left(e\right)\left(g_{n}-e_{n}\right)$ is invertible with $\|T\|,\|T^{-1}\|\le \frac{1}{1-a}$, and $Te_{n}=g_{n}$, for every $n\in\N$, and in particular $\left\{g_{n}\right\}_{n\in\N}$ is basic.\medskip
\item[(iii)] If  $\left\{h_{i}\right\}_{i\in I}\subset \overline{\spa}\left\{g_{n}\right\}_{n\in\N}$ is normalized and weakly null, then there is a basic sequence $\left\{h_{i_{k}}\right\}_{k\in\N}$ and an increasing sequence of indices $\left\{m_{k}\right\}_{k\in\N}$ such that\linebreak $d\left(h_{i_{k}},\spa\left\{e_{n}\right\}_{n=m_{k}+1}^{m_{k+1}}\right)\xrightarrow[k\to\8]{}0$. If additionally $\left\{h_{i}\right\}_{i\in I}$ is null with respect to a semi-norm $\rho$, then $\left\{h_{i_{k}}\right\}_{k\in\N}$ can also be selected $\rho$-null.\medskip
\item[(iv)] If $H$ is a closed subspace of $E$ such that $\dim E_{1} \cap H<\8$, then there is $m\in\N$ such that $E_{m}\cap H=\left\{0_{E}\right\}$. Moreover, if $E_{1}+ H$ is closed, the same is true for $E_{m}+ H$.
\end{theorem}

It turns out that the class of operators which are complementary to the weak topology is one of the familiar classes of operators.

\begin{theorem}\label{usf}For a continuous linear operator $T$ between Banach spaces $E$ and $F$ the following conditions are equivalent:
\item[(i)] $T$ is upper semi-Fredholm (USF), i.e. it has a finitely dimensional kernel and a closed range;
\item[(ii)] There is a finitely co-dimensional subspace of $E$, on which $T$ is bounded from below;
\item[(iii)] $\dim\Ker \left(T-S\right)<\8$, for every compact operator $S:E\to F$, i.e. there is no compact operator that coincides with $T$ on an infinitely dimensional subspace;
\item[(iv)] No restriction of $T$ to an infinitely dimensional subspace is compact;
\item[(v)] $\left.T\right|_{\Bbo_{E}}$ is proper, i.e. $T^{-1}\left(K\right)\cap \Bbo_{E}$ is compact, for every compact $K\subset F$;
\item[(vi)] $T\ne 0$ and $\left.T\right|_{\Bbo_{E}}$ is closed, i.e. $T\left(A\right)$ is closed, for every closed $A\subset \Bbo_{E}$;
\item[(vii)] If $\left\{e_{n}\right\}_{n\in\N}\subset \So_{E}$ is such that $\left\{Te_{n}\right\}_{n\in\N}$ is convergent, then $\left\{e_{n}\right\}_{n\in\N}$ contains a convergent subsequence;
\item[(viii)] If a normalized sequence is $\rho_{T}$-null, it contains a convergent subsequence;
\item[(ix)] No normalized basic sequence is $\rho_{T}$-null;
\item[(x)] If $\left\{e_{n}\right\}_{n\in\N}$ is basic, then there is $m\in\N$ such that $T$ is bounded from below on $\overline{\spa}\left\{e_{n}\right\}_{n\ge m}$;
\item[(xi)] If $\left\{e_{n}\right\}_{n\in\N}$ is basic, then the restriction of $T$ to $\overline{\spa}\left\{e_{n}\right\}_{n\in\N}$ is not compact;
\item[(xii)] Whenever $H$ is a normed space and $S\in\Lo\left(H,E\right)$ are such that $TS$ is compact, $S$ is compact;
\item[(xiii)] $T$ is complementary to the weak topology on $E$.
\end{theorem}
\begin{proof}
\textbf{Part 1.} (iii)$\Leftrightarrow$(iv)$\Leftrightarrow$(ix)$\Leftrightarrow$(xi)$\Leftrightarrow$(xiii). First, we obviously have (iv)$\Rightarrow$(iii),(xi), and (ix)$\Rightarrow$(xiii) follows from part (iii) of Theorem \ref{basic}.\medskip

(xiii)$\Rightarrow$(iv):  If a restriction of $T$ to a closed subspace $G$ is compact, then $\rho_{T}$ is continuous with respect to the weak topology on $\Bbo_{G}$. Hence, from part (ii) of Proposition \ref{compl1}, the norm and weak topologies coincide on $\Bbo_{G}$, and so $\dim G<\8$.\medskip

(iii),(xi)$\Rightarrow$(ix): If there is a $\rho_{T}$-null normalized basic sequence, by passing to a subsequence we can find a basic $\left\{e_{n}\right\}_{n\in\N}\subset\So_{E}$ such that $\sum\limits_{n\in\N}\|Te_{n}\|<+\8$. Then it follows from part (i) of Theorem \ref{basic} that $\left.T\right|_{G}$ is compact, where $G=\overline{\spa}\left\{e_{n}\right\}_{n\in\N}$.\bigskip

\textbf{Part 2.} (i)$\Leftrightarrow$(ii)$\Rightarrow$(v)$\Leftrightarrow$(vi)$\Leftrightarrow$(vii)$\Rightarrow$(viii)$\Rightarrow$(ix). First, (i)$\Leftrightarrow$(ii) and (v)$\Leftrightarrow$(vii) are standard, (vii)$\Rightarrow$(viii) is trivial, and (viii)$\Rightarrow$(ix) follows from the fact that a normalized basic sequence is a discrete closed set. Let $H=\Ker T$.\medskip

(v)$\Rightarrow$(vi) for general maps see in \cite[3.7.18]{engelking}. To get the converse we need to show that the preimages of the singletons are compact. In the present setting this amounts to showing that $\dim H<\8$. If this is not true, there is a sequence $\left\{e_{n}\right\}_{n\in\N}\subset \frac{1}{2}\Bbo_{H}$ with no convergent subsequences. Then, $G=\left\{e_{n}+\frac{1}{n}f\right\}_{n\in\N}$ also has no convergent subsequences, where $f\not\in H$ with $\|f\|=\frac{1}{2}$. But then $G$ is a closed subset of $\Bbo_{E}$ whereas $TG=\left\{\frac{1}{n}Tf\right\}_{n\in\N}$ is not closed. Contradiction.\medskip

(i)$\Rightarrow$(v): Let $G$ be a (finitely co-dimensional) subspace of $E$ such that $E=H+G$. Note that $E\simeq G\oplus_{\8}H$, and so $\Bbo_{E}\subset r\Bbo_{G}+r\Bbo_{H}$, for some $r>0$. If $K\subset F$ is compact, then $T^{-1}\left(K\right)\cap \Bbo_{E}\subset \left(\left.T\right|_{G}^{-1}\left(K\right)\cap r\Bbo_{E}\right) + r\Bbo_{H}$. Since $\left.T\right|_{G}$ is a homeomorphism onto its image, the first set in the sum is compact. Since $\dim H<+\8$, the second set in the sum is also compact, and so $T^{-1}\left(K\right)\cap \Bbo_{E}$ is compact.\medskip

\textbf{Part 3.} Leftovers. First, (x)$\Rightarrow$(xi) is obvious; (xii)$\Rightarrow$(iv) follows from considering the inclusion operator of a subspace of $E$.

(v)$\Rightarrow$(xii): We may assume that $\|S\|=1$. Since $TS$ is a compact operator, $\overline{TS \Bo_{H}}$ is compact in $F$, from where $S\Bo_{H}\subset \Bbo_{E}\cap T^{-1}\overline{TS \Bo_{H}}$. Therefore, $S\Bo_{H}$ is relatively compact, and so $S$ is a compact operator.\medskip

(i)$\Rightarrow$(x): Since $\dim H<\8$, from part (iv) of Theorem \ref{basic} there is $m\in\N$ such that $H\cap\overline{\spa}\left\{e_{n}\right\}_{n\ge m}=\left\{0_{E}\right\}$. Hence, $T$ is an injection on $\overline{\spa}\left\{e_{n}\right\}_{n\ge m}$ and has a closed range. Thus, $T$ is bounded from below on $\overline{\spa}\left\{e_{n}\right\}_{n\ge m}$.\medskip

(viii)+(iv)$\Rightarrow$(i): First, (iv) trivially implies that $\dim H<\8$. Let $G$ be a (finitely co-dimensional) subspace of $E$ such that $E=H+G$. Since $\left.T\right|_{G}$ is an injection, $TE=TG$ is closed if and only if $\left.T\right|_{G}$ is bounded from below. Assume that this is not the case. Then there is a $\rho_{T}$-null sequence $\left\{g_{n}\right\}_{n\in\N}\subset \So_{G}$. From (viii), by passing to a subsequence, we may assume that $g_{n}\xrightarrow[n\to\8]{}g\in \So_{G}$. But then $Tg=\lim\limits_{n\to\8} Tg_{n}=0_{F}$, which contradicts injectivity of $T$ on $G$.\medskip

(ix)+(xiii)$\Rightarrow$(viii): Let $\left\{e_{n}\right\}_{n\in\N}\subset \So_{E}$ be $\rho_{T}$-null. From (ix) this sequence does not contain basic subsequences, and so is weakly relatively compact (see \cite[Theorem 1.5.6]{ak}). By passing to a subsequence we may assume that $e_{n}\xrightarrow[n\to\8]{w}e\in E$. For any $\nu\in F^{*}$ we have that $$\left|\nu\left(Te\right)\right|=\left|T^{*}\nu\left(e\right)\right|=\lim\limits_{n\to\8}\left|T^{*}\nu\left(e_{n}\right)\right|=\lim\limits_{n\to\8}\left|\nu\left(Te_{n}\right)\right|\le \|\nu\|\lim\limits_{n\to\8}\|Te_{n}\|=0,$$ from where $Te=0_{F}$. But then $\left\{e_{n}-e\right\}_{n\in\N}$ is simultaneously $\rho_{T}$-null and weakly null, and so from (xiii) $e_{n}\xrightarrow[n\to\8]{}e$.
\end{proof}

\begin{remark}
Many of the conditions considered in the theorem are well-known, see e.g. \cite{aiena}, \cite{gm}, \cite{sing2}. It would be desirable to get a direct proof of (xiii)$\Rightarrow$(v) in the light of condition (iv) in Theorem \ref{compl}.
\end{remark}

It is clear that any bounded from below operator is USF, and that a restriction of a USF operator to an infinitely dimensional subspace of $E$ is USF. Combining Proposition \ref{usf} with Remark \ref{complo2} we get the following corollaries.

\begin{corollary}
Let $T$ be a continuous linear operator between Banach spaces $E$ and $F$. Then the collection of subspaces of $E$ on which $T$ is USF is gap-open.
\end{corollary}

\begin{corollary}\label{usf1}
For a fixed subspace $H$ of a Banach space $E$, the collection of operators which are USF on $H$ is open.
\end{corollary}

In a reflexive space USF operators are characterized by a stronger condition than (ix) in Theorem \ref{usf}.

\begin{theorem}\label{ref}$E$ is reflexive if and only if the USF property for $T\in\Lo\left(E,F\right)$ is equivalent to existence of $\delta>0$ such that no normalized basic sequence $\left\{e_{n}\right\}_{n\in\N}$ satisfies $\rho_{T}\left(e_{n}\right)<\delta$.
\end{theorem}
\begin{proof}
Necessity: It is clear that the the existence of $\delta$ as in the proposition is stronger than the condition (ix) in Theorem \ref{usf}. On the other hand, in a reflexive Banach space every basic sequence is weakly null (in fact, it is shrinking which is a much stronger condition, see \cite[Section 3.2]{ak}), and so combining  the condition (xii) in Theorem \ref{usf} with the condition (v) in Theorem \ref{compl} guarantees existence of $\delta$.\medskip

Sufficiency: Assume that $E$ is not reflexive and take any $e\in \So_{E}$. Let $F$ be a closed subspace of co-dimension $1$ that does not include $e$. Without loss of generality we may assume that $E=\spa\left\{e\right\}\oplus_{\8} F$. Since $F$ is not reflexive it follows from \cite{sing1} that there is a normalized basic sequence $\left\{f_{n}\right\}_{n\in\N}\subset F$ of type P*, i.e. such that $\left|a_{1}+...+a_{n}\right|\le r\|a_{1}f_{1}+...+a_{n}f_{n}\|$, for some $r>0$ and any $a_{1},...,a_{n}\in\R$.

For $t>0$ let $e_{n}= e+t f_{n}$; then $\|e_{n}\|=t+1$, for every $n\in\N$. If $P$ is the (USF) projection from $E=\spa\left\{e\right\}\oplus_{\8} F$ onto $F$, then $\|P\frac{e_{n}}{t+1}\|=\frac{t}{t+1}$, for every $n\in\N$. Since $\left\{\frac{e_{n}}{t+1}\right\}_{n\in\N}$ is a normalized sequence, and $t$ can be taken arbitrarily small, it is left to show that $\left\{e_{n}\right\}_{n\in\N}$ is a basic sequence. Let $R$ be the basic constant of $\left\{f_{n}\right\}_{n\in\N}$. For every $a_{1},...,a_{n}\in\R$ and $m<n$ we have
\begin{align*}
\|a_{1}e_{1}+...+a_{m}e_{m}\|&= \left|a_{1}+...+a_{m}\right|+t\|a_{1}f_{1}+...+a_{m}f_{m}\|\\
&\le \left(r+t\right)\|a_{1}f_{1}+...+a_{m}f_{m}\|\le \left(r+t\right)R\|a_{1}f_{1}+...+a_{n}f_{n}\|\\
&\le \frac{\left(r+t\right)R}{t}\|a_{1}e_{1}+...+a_{n}e_{n}\|,
\end{align*}
and so $\left\{e_{n}\right\}_{n\in\N}$ is a basic sequence with the basis constant at most $\frac{\left(r+t\right)R}{t}$.
\end{proof}

The class of operators that is in a way opposite to the USF, and generalize compact operators are \emph{strictly singular} (SS) operators.

\begin{proposition}\label{ss}For a continuous linear operator $T$ between Banach spaces $E$ and $F$ the following conditions are equivalent:
\item[(i)] $T$ is SS, i.e. no restriction of $T$ to an infinitely dimensional subspace of $E$ is bounded from below;
\item[(ii)] No restriction of $T$ to an infinitely dimensional subspace of $E$ is USF;
\item[(iii)] Every infinitely dimensional closed subspace $H$ of $E$ contains an infinitely dimensional closed sub-subspace $G\subset H$ such that $\left.S\right|_{G}$ is compact;
\item[(iv)] Every infinitely dimensional closed subspace $H$ of $E$ contains a normalized weakly null $\rho_{S}$-null net;
\item[(v)] Every infinitely dimensional closed subspace $H$ of $E$ contains a normalized basic $\rho_{S}$-null net.
\end{proposition}
\begin{proof}
(ii)$\Rightarrow$(i) is trivial, while the converse follows from the fact that a USF operator is bounded from below on a finitely co-dimensional subspace.

Equivalence of each of (iii), (iv) and (v) with (ii) follows from the equivalent characterizations of USF operators in Theorem \ref{usf}
\end{proof}

It is clear that a restriction of a SS operator to an infinitely dimensional subspace of $E$ is SS, from where one can deduce that a linear combination of SS operators is SS. Since for every subspace $H$, the set of operators which are USF on $H$ is open, it follows that SS operators form a closed subspace of $\Lo\left(E,F\right)$. It is easy to see that if $T$ is USF, and $S$ is SS, then $T-S$ is USF, and no restriction of $T$ to an infinitely dimensional subspace of $E$ is SS.\medskip

The fact that USF operators are a specific case of operators complementary to a given topology, it is natural to introduce classes of operators analogous to several classes associated to USF operators. Let $P_{\tau}\left(E,F\right)$ be the set of all operators $S\in\Lo\left(E,F\right)$ such that $T+S\in U_{\tau}\left(E,F\right)$, for every $T\in U_{\tau}\left(E,F\right)$. Let $I_{\tau}\left(E,F\right)$ be the set of all operators $S\in\Lo\left(E,F\right)$ such that $Id_{E}-TS\in U_{\tau}\left(E,F\right)$, for every $T\in \Lo\left(F,E\right)$. The following is proven similarly to \cite[theorems 7.21, 7.46]{aiena}.

\begin{proposition}
$P_{\tau}\left(E,F\right)\subset I_{\tau}\left(E,F\right)$ is a closed subspace of $\Lo\left(E,F\right)$, such that if $G$ is a normed space, $R_{1}\in \Lo\left(E\right)$, $R_{2}\in \Lo\left(F,G\right)$, and $S\in P_{\tau}\left(E,F\right)$, then $R_{1}SR_{2}\in P_{\tau}\left(E,G\right)$.
\end{proposition}

\section{Dispersed subspaces}\label{dispe}

Everywhere in this section $F$ is a Banach lattice. A subspace $E$ of a Banach lattice $F$ is called \emph{dispersed} if it does not contain an \emph{almost disjoint sequence}, i.e. a sequence $\left\{e_{n}\right\}_{n\in\N}\subset \So_{E}$, for which there is a disjoint $\left\{f_{n}\right\}_{n\in\N}\subset F$ such that $\|f_{n}-e_{n}\|\xrightarrow[n\to\8]{} 0$. This condition is equivalent to the fact that if $\left\{f_{n}\right\}_{n\in\N}\subset\So_{F}$ is disjoint, then $\liminf\limits_{n\to\8}\rho_{E}\left(f_{n}\right)>0$. Let us introduce variations of this concept.\medskip

We will call $E$ \emph{strictly dispersed} if there is $\delta>0$ such that there is no disjoint $\left\{f_{n}\right\}_{n\in\N}\subset\So_{F}$ at the distance less than $\delta$ from $E$. It is clear that every strictly dispersed subspace is dispersed. Using metric $d_{3}$ on the set of the subspaces it is easy to show that the set of strictly dispersed subspaces of $F$ is gap-open.\medskip

Recall that the \emph{unbounded norm (un-) topology} on a Banach lattice is the linear topology determined by the neighborhoods of zero of the form $\left\{f\in F, \|\left|f\right|\wedge h\|<\varepsilon \right\}$, where $h\in F_{+}$ and $\varepsilon>0$. More information about the un-topology see in \cite{dot}, \cite{klt}, \cite{kmt}. We also introduce the (weaker) \emph{una-topology}, which is determined by the neighborhoods of zero of the form $\left\{f\in F, \|\left|f\right|\wedge h\|<\varepsilon \right\}$, where $h\in F^{a}_{+}$ and $\varepsilon>0$ ($F^{a}$ stands for the closed ideal of the order continuous elements of $F$, see \cite[Proposition 2.4.10]{mn}). As a disjoint order-bounded sequence in an order continuous Banach lattice is null, it follows that any disjoint sequence is una-null. Since a Banach lattice is order continuous if and only if $F^{a}=F$, in the order continuous Banach lattices the un- and una-topologies coincide. It is not hard in fact to show that the converse is also true (an order bounded non-null disjoint sequence is una-null, but not un-null).

We will call $E$ \emph{weakly/strongly dispersed} if $0_{F}$ is separated from $\So_{E}$ with respect to the un-/una-topology. Obviously, these conditions coincide if $F$ is order continuous. The following is a particular case of Proposition \ref{compl2} and Corollary \ref{open}.

\begin{proposition}\label{ws} The set of weakly/strongly dispersed subspaces is gap-open. Moreover, for a closed subspace $E$ of $F$ the following conditions are equivalent:
\item[(i)] $E$ is weakly/strongly dispersed;
\item[(ii)] The un-/una-topology coincides with the norm topology on $E$;
\item[(iii)] $\rho_{E}$ and the un-/una-topology are complementary on $F$.
\end{proposition}

Recall that from Kadec-Pelczynski theorem, we can select an almost disjoint sequence from any un-null net (see \cite[Theorem 3.2]{dot}). Hence, every dispersed subspace is weakly dispersed. Note that unless $F$ is order continuous, the converse is false: the span of an order bounded non-null disjoint sequence is trivially not dispersed, but the un-topology coincides with the norm topology on it, as it has a strong unit (\cite[Theorem 2.3]{kmt}).

In a discrete order continuous Banach lattice the un-topology is weaker than the weak topology (see \cite[Proposition 4.16]{kmt}). If $E$ is dispersed in such a lattice, it is weakly dispersed, and so the weak topology has to be stronger than the norm topology on $E$, from where $\dim E<\8$.

\begin{corollary}[\cite{gmm}]\label{doc}In a discrete order continuous Banach lattice dispersed subspaces are finite-dimensional.
\end{corollary}

While the weak disperseness is clearly the weakest of the introduced notions of disperseness, the strong disperseness is the strongest.

\begin{proposition}\label{str}If $E$ is strongly dispersed, then $E$ is strictly dispersed.
\end{proposition}

We will later prove a more general fact (see part (ii) of Proposition \ref{sdns}), but let us prove the proposition separately and using a different method. Let us start with a lemma.

\begin{lemma}\label{six}
Let $E$ be a subspace of $F$ such that there is a normalized disjoint sequence at the distance at most $\varepsilon$ from $E$. Then, there is $\left\{e_{n}\right\}_{n\in\N}\subset\So_{E}$ such that $\|\left|e_{m}\right|\wedge \left|e_{n}\right|\|\le 6\varepsilon$, for distinct $m,n\in\N$.
\end{lemma}
\begin{proof}
Let $\left\{f_{n}\right\}_{n\in\N}\subset\So_{F}$ be a disjoint sequence at the distance less than $\varepsilon$ from $E$. From Lemma \ref{sph} for every $n\in\N$ there is $e_{n}\in \So_{E}$ such that $\|f_{n}-e_{n}\|\le 2\varepsilon$. Then for distinct $m,n\in\N$ we have
\begin{align*}\left|e_{m}\right|\wedge \left|e_{n}\right|&\le \left(\left|f_{m}\right|+\left|f_{m}-e_{m}\right|\right)\wedge \left(\left|f_{n}\right|+\left|f_{n}-e_{n}\right|\right)
\\&\le \left|f_{m}\right|\wedge \left|f_{n}\right|+\left|f_{m}\right|\wedge \left|f_{n}-e_{n}\right|+\left|f_{m}-e_{m}\right|\wedge \left|f_{n}\right|+\left|f_{m}-e_{m}\right|\wedge \left|f_{n}-e_{n}\right|\\&\le \left|f_{m}-e_{m}\right|+\left|f_{n}-e_{n}\right|+\left|f_{m}-e_{m}\right|\wedge \left|f_{n}-e_{n}\right|.\end{align*}
The norm of each of the summands in the left-hand side is at most $2\varepsilon$, from where the result follows.
\end{proof}

\begin{proof}[Proof of Proposition \ref{str}]
Since the una-topology coincides with the norm topology on $E$, there are $h\in F^{a}_{+}$ and $\delta>0$ be such that $\|h\wedge\left|e\right|\|\ge \delta$, for every $e\in \So_{E}$. Since the ideal generated by $h$ is order continuous and has a strong unit, there is an AL norm $\|\cdot\|_{1}\le \|\cdot\|$ on it (see \cite[Theorem 2.7.8]{mn}). Moreover, from Amemiya's theorem this norm induces the same topology on $\left[0,h\right]$ (see \cite[Theorem 2.4.8]{mn}), and so there is $\delta_{1}>0$ such that $\|h\wedge\left|e\right|_{1}\|\ge \delta_{1}$, for every $e\in \So_{E}$.

Let $\varepsilon>0$ be such that there is a normalized disjoint sequence at the distance at most $\varepsilon$ from $E$. Our goal is to get a uniform lower bound for $\varepsilon$. From Lemma \ref{six} there is $\left\{e_{n}\right\}_{n\in\N}\subset\So_{E}$ such that $\|\left|e_{m}\right|\wedge \left|e_{n}\right|\|\le 6\varepsilon$, for distinct $m,n\in\N$. For $n\in\N$ denote $h_{n}=\left|e_{n}\right|\wedge h$; we have $\|h_{n}\|_{1}\ge\delta_{1}$, while for distinct $m,n\in\N$ we have $\|h_{m}\wedge h_{n}\|_{1}\le\|h_{m}\wedge h_{n}\|\le\|\left|e_{m}\right|\wedge \left|e_{n}\right|\|\le 6\varepsilon$. Note that $$\bigvee\limits_{k=1}^{n+1}h_{k}-\bigvee\limits_{k=1}^{n}h_{k}=h_{n+1}-h_{n+1}\wedge\bigvee\limits_{k=1}^{n}h_{k}\ge h_{n+1}-\sum\limits_{k=1}^{n}h_{n+1}\wedge h_{k},$$ from where, since $\|\cdot\|_{1}$ is an AL norm, we get $$\left\|\bigvee\limits_{k=1}^{n+1}h_{k}\right\|_{1}-\left\|\bigvee\limits_{k=1}^{n}h_{k}\right\|_{1}\ge \left\|h_{n+1}\right\|_{1}-\sum\limits_{k=1}^{n}\left\|h_{n+1}\wedge h_{k}\right\|_{1}\ge \delta_{1}-6n\varepsilon.$$ Therefore, recursively we can show that $\|h\|_{1}\ge \left\|\bigvee\limits_{k=1}^{n}h_{k}\right\|_{1}\ge n\delta_{1} - 3\varepsilon n\left(n+1\right)$. Thus, $\varepsilon\ge \frac{n\delta_{1}-\|h\|_{1}}{3n\left(n+1\right)}$, for every $n\in\N$, and so $\varepsilon\ge \sup\limits_{n\in\N}\frac{n\delta-\|h\|_{1}}{3n\left(n+1\right)}>0$, where the last quantity depends only on $h$.
\end{proof}

Let us summarize: strongly dispersed, strictly dispersed, dispersed and weakly dispersed subspaces of $F$ form collections, listed in the order of inclusion. But in an order continuous Banach lattice the notions of strong and weak disperseness coincide, and so all four notions are the same, which is the main result of this section.

\begin{theorem}\label{disp}
In an order continuous Banach lattice the notions of strong, strict and weak disperseness coincide with the dispreseness. In particular, the set of dispersed subspaces of an order continuous Banach lattice is gap-open.
\end{theorem}

Let us turn to dispersed subspaces in non-order continuous Banach lattices. As was mentioned above, not every weakly dispersed subspace is dispersed; let us show that the implication strongly dispersed$\Rightarrow$strictly dispersed is also proper.

\begin{example}\label{rad}
It is not hard to show that if $F=L_{\8}\left[0,1\right]$, then $F^{a}=\left\{0_{F}\right\}$, and so the una-topology is trivial. Hence, there are no strongly dispersed subspaces of $F$. However, we will show that $F$ contains a strictly dispersed subspace.

First, observe that if $f,g\in F$ are independent as random variables, and symmetrically distributed (in the sense that the distribution of $f$ and $-f$ is the same), then $\|f+g\|=\|f\|+\|g\|$, as for every $\varepsilon>0$ the set $f^{-1}\left(\|f\|-\varepsilon, \|f\|\right)\cap g^{-1}\left(\|g\|-\varepsilon, \|g\|\right)$ is non-negligible. In a similar way one can extend this fact for an arbitrary finite collection of independent functions. Moreover, if $h$ is independent of the vector $\left(f,g\right)$ and symmetrically distributed, then $\|\left|f\right|\wedge \left|g+h\right|\|\ge \|f\|\wedge \|h\|$, since if e.g. $f^{-1}\left(\|f\|-\varepsilon, \|f\|\right)\cap g^{-1}\left[0,+\8\right)$ is non-negligible, then it has a non-negligible intersection with the set $h^{-1}\left(\|h\|-\varepsilon, \|h\|\right)$.\medskip

Let $\left\{r_{n}\right\}_{n\in\N}$ be the sequence of Rademacher functions, i.e. independent random variables which attain values $\pm 1$ with probability $\frac{1}{2}$. Assume that there is a disjoint sequence at the distance $\varepsilon<\frac{1}{18}$ from $E=\overline{\spa}\left\{r_{n}\right\}_{n\in\N}$. Then, from Lemma \ref{six} there is a normalized sequence $\left\{e_{n}\right\}_{n\in\N}\subset\spa\left\{r_{n}\right\}_{n\in\N}$ such that $\|\left|e_{n}\right|\wedge \left|e_{m}\right|\|\le 6\varepsilon$.\medskip

For every $n\in\N$ let $A_{n}=A_{n}^{+}\sqcup A_{n}^{-}\subset \N$ be such that there are $\left\{a_{in}\right\}_{n\in\N,~ i\in A_{n}}\subset\left(0,1\right]$ such that $e_{n}=\sum\limits_{i\in A_{n}^{+}}a_{in}r_{i}-\sum\limits_{i\in A^{-}_{n}}a_{in}r_{i}$. Note that $A_{n}\cap A_{m}\ne\varnothing$, since otherwise $e_{n}$ and $e_{m}$ would be independent, and so $\|\left|e_{n}\right|\wedge \left|e_{m}\right|\|=1$. Since $A_{1}$ is finite, there are $m,n\in\N$ such that $A_{1}\cap A^{\pm}_{n}=A_{1}\cap A^{\pm}_{m}$.  Then, $e_{m}= \sum\limits_{i\in A_{m}^{+}\cap A_{1}}a_{im}r_{i}-\sum\limits_{i\in A^{-}_{m}\cap A_{1}}a_{im}r_{i}+h_{m}$, where $h_{m}$ is independent of $\left\{r_{i}\right\}_{i\in A_{1}\cap A_{m}}$. Hence, $6\varepsilon> \|\left|e_{1}\right|\wedge \left|e_{m}\right|\|\ge \|e_{1}\|\wedge \|h_{m}\|=1\wedge \|h_{m}\|$, and so $\|h_{m}\|<6\varepsilon$. Therefore, since $1=\|e_{m}\|=\sum\limits_{i\in A_{1}\cap A_{m}} a_{im}+\|h_{m}\|$, it follows that $\sum\limits_{i\in A_{1}\cap A_{m}} a_{im}>1-6\varepsilon$. Hence, on $\bigcap\limits_{i\in A_{1}\cap A_{m}^{\pm}}r_{i}^{-1}\left(\pm1\right)$ the value of $e_{m}$ is at least $\sum\limits_{i\in A_{1}\cap A_{m}} a_{im}-\|h_{m}\|>1-12\varepsilon$, and similarly, the value of $e_{n}$ on this set is also at least $1-12\varepsilon$. Thus, $1-12\varepsilon\le\|\left|e_{n}\right|\wedge \left|e_{m}\right|\| <6\varepsilon$. Contradiction.
\qed\end{example}

\begin{question} Is the set of dispersed subspaces of $F$ always gap-open?
\end{question}

Some evidence in support of the hypothesis is the characterization from Proposition \ref{sumcl} below and results such as \cite[Chapter IV, Theorem 4.30]{kato}. Moreover, it is plausible that every dispersed subspace is strictly dispersed. A competing hypothesis may be that the set of strictly dispersed subspaces is the gap interior of the set of the dispersed subspaces, and the two notions coincide if and only if $F$ is order-continuous.

\begin{lemma}\label{quas} Let $\left\{f_{n}\right\}_{n\in\N}\subset \So_{F}$ be quasi-disjoint, i.e. there is a disjoint sequence $\left\{e_{n}\right\}_{n\in\N}\subset \So_{F}$ such that $\sum\limits_{n\in\N}\|f_{n}-e_{n}\|<1$. Then:
\item[(i)] For every $\varepsilon>0$ there is $m\in\N$ and an isomorphism $T\in\Lo\left(F\right)$ such that $\|T\|,\|T^{-1}\|<1+\varepsilon$, $Id_{F}-T$ is compact and $Te_{n}=f_{n}$, for every $n\ge m$.
\item[(ii)] If $\left\{h_{i}\right\}_{i\in I}$ is a normalized weakly null net in $\overline{\spa}\left\{f_{n}\right\}_{n\in\N}$, then it contains an almost disjoint sequence $\left\{h_{i_{k}}\right\}_{k\in\N}$. If additionally $\left\{h_{i}\right\}_{i\in I}$ is null with respect to a semi-norm $\rho$, then $\left\{h_{i_{k}}\right\}_{k\in\N}$ can also be selected $\rho$-null.
\end{lemma}
\begin{proof}
(i) follows from part (ii) of Theorem \ref{basic} and the fact that a disjoint sequence is always basic with basis constant $1$.

(ii): From part (iii) of Theorem \ref{basic} there is a ($\rho$-null) subsequence $\left\{h_{i_{k}}\right\}_{k\in\N}$ and an increasing sequence of indices $\left\{m_{k}\right\}_{k\in\N}$, such that $\|h_{i_{k}}-g_{k}\|\xrightarrow[k\to\8]{}0$, where $g_{k}\in \spa\left\{e_{n}\right\}_{n=m_{k}+1}^{m_{k+1}}$, for every $k\in\N$. But since $\left\{e_{n}\right\}_{n=m_{k}+1}^{m_{k+1}}\perp \left\{e_{n}\right\}_{n=m_{l}+1}^{m_{l+1}}$, for $k<l$, it follows that $\left\{g_{k}\right\}_{k\in\N}$ is disjoint, and so $\left\{h_{i_{k}}\right\}_{k\in\N}$ is almost disjoint.
\end{proof}

The next result follows from part (i) Lemma \ref{quas} and the fact that a closed span of a disjoint sequence in $l_{\8}$ is isometrically isomorphic to $c_{0}$.

\begin{proposition}\label{wsd} If $E\subset l_{\8}$ is not dispersed then it contains an isomorphic copy of $c_0$ which is arbitrary close to being an isometric copy.
\end{proposition}

\begin{example}
Consider a Rademacher type sequence $\left\{g_{n}\right\}_{n\in\N}\subset F=l_{\8}$, where $g_{1}=\left(1,-1,1,-1,...\right)$, $g_{2}=\left(1,1,-1,-1,1,1,...\right)$ and so on. Similarly to Example \ref{rad} one can show that  $E=\overline{\spa}\left\{g_{n}\right\}_{n\in\N}$ is strictly dispersed in $F$. Note that unlike $L_{\8}\left[0,1\right]$, the una-topology on $l_{\8}$ is Hausdorff. In fact, it is the topology of the coordinate-wise convergence. Hence, every dispersed subspace of $c_{0}=F^{a}$ is strongly dispersed. On the other hand, $E$ is not strongly dispersed, since $\left\{h_{m}\right\}_{m\in\N}\subset\So_{E}$ is null coordinate-wise, where $h_{m}=\frac{1}{m}\left(g_{1}-g_{2}+...+\left(-1\right)^{m+1}g_{m}\right)$. Indeed, for every $n$, the $n$-th coordinate of $g_{m}$ is $1$ as soon as $2^{m}\ge n$, and so the $n$-th coordinate of $h_{m}$ is either $\frac{k}{m}$, or $\frac{k+1}{m}$, where $k\in\Z$ does not depend on $m$.
\qed\end{example}\medskip

Let us call a subspace $E$ of $F$ \emph{anti-dispersed} if it contains no infinitely dimensional dispersed sub-subspaces, i.e. every infinitely dimensional sub-subspace of $E$ contains an almost disjoint sequence. Every subspace of an anti-dispersed subspace is trivially anti-disperse. Since every dispersed subspace has to be weakly dispersed, we get the following result.

\begin{proposition}\label{weak}If $E$ is such that the weak topology is stronger than the un-topology on $E$, then $E$ is anti-dispersed.
\end{proposition}

It is easy to see that a closed span of a disjoint sequence is anti-dispersed. However, from part (ii) of Lemma \ref{quas} this remains true for a sequence which is sufficiently close to being disjoint.

\begin{proposition}\label{anti} If $\left\{f_{n}\right\}_{n\in\N}\subset \So_{F}$ is quasi-disjoint, then $\overline{\spa}\left\{f_{n}\right\}_{n\in\N}$ is anti-dispersed.
\end{proposition}

Since every almost disjoint sequence contains a quasi-disjoint one, we get the following corollary.

\begin{corollary}\label{quasi}Every non-dispersed subspace contains an anti-dispersed sub-subspace.
\end{corollary}

\begin{question}
Does Proposition \ref{anti} remain true for any almost disjoint sequence? Is every anti-dispersed subspace contained in a closed span of an (almost) disjoint sequence? Is the converse to Proposition \ref{weak} true?
\end{question}

\begin{question}
Is the set of anti-dispersed subspaces gap-closed? Equivalently, if a subspace $E$ contains a dispersed subspace $H$, is there a gap-neighborhood of $E$, such that every $G$ in this neighborhood contains a subspace which is sufficiently gap-close to $H$?
\end{question}

In regards to the last question, the answer is positive if $E$ is complemented (this follows from \cite[Theorem 5.2]{ostrov}). Not every dispersed subspace is complemented (combine \cite[Proposition 5.7.1 and Theorem 7.6]{ak}, and the fact from \cite{gmm} that a subspace is strongly embedded if and only if it is dispersed), but they are in some classes of subspaces (see \cite[Theorem 6.4.8]{ak}), which gives a reason to think that this property is somewhat close to being complemented.

\begin{question}
Can we find a topology such that dispersed or strictly dispersed subspaces are characterized similarly to Proposition \ref{ws}?
\end{question}

Note that we need this topology to simultaneously satisfy the following properties: every bounded disjoint sequence has to be null and every bounded null net has to contain an almost bounded sequence. We cannot expect this topology to be linear. Indeed, consider $F=l_{\8}=l_{\8}\left(\N\times \N\right)$, $e_{n}=\left\{\delta_{mn}\right\}_{m,k\in \N}$, $f_{n}=\left\{\delta_{nk}\right\}_{m,k\in \N}$ and $g_{n}=e_{n}+f_{n}$, for $n\in\N$. Then, $\left\{e_{n}\right\}_{n\in\N}$ and $\left\{f_{n}\right\}_{n\in\N}$ are disjoint, but $\|g_{n}\wedge g_{m}\|=1$ for any distinct $m,n\in\N$, and so from Lemma \ref{six} no disjoint sequence is at the distance less than $\frac{1}{6}$ from $\left\{g_{n}\right\}_{n\in\N}$.

\section{DNS and DSS operators}\label{dnsdss}

In this section we will focus on the two classes of operators acting from a Banach lattice into a Banach space, which are analogous to the upper semi-Fredholm and strictly singular operators. In a way, the un-topology, almost disjoint sequences and dispersed subspaces play similar roles in regards to the former classes of operators as the weak topology, normalized basic sequences and finitely dimensional subspaces play in regards to the latter. Everywhere in this section $F$ is again a Banach lattice, and $E$ is a Banach space.

An operator $T:F\to E$ is called \emph{disjointly non-singular (DNS)} operator if one of the following equivalent conditions is satisfied (equivalence of (i)-(vi) was established in \cite{gmm}):

\begin{theorem}\label{dns} For $T\in\Lo\left(F,E\right)$ the following conditions are equivalent:
\item[(i)] No (almost) disjoint sequence in $\So_{F}$ is null with respect to $\rho_{T}$;
\item[(ii)] If $\left\{f_{n}\right\}_{n\in\N}\subset  F$ is disjoint, then the restriction of $T$ is not compact on $\overline{\spa}\left\{f_{n}\right\}_{n\in\N}$;
\item[(iii)] If $\left\{f_{n}\right\}_{n\in\N}\subset  F$ is disjoint, then the restriction of $T$ is not SS on $\overline{\spa}\left\{f_{n}\right\}_{n\in\N}$;
\item[(iv)] If $\left\{f_{n}\right\}_{n\in\N}\subset F$ is disjoint, then the restriction of $T$ is USF on $\overline{\spa}\left\{f_{n}\right\}_{n\in\N}$;
\item[(v)] If $\left\{f_{n}\right\}_{n\in\N}\subset F$ is disjoint, then there is $m\in\N$ such that  $T$ is bounded from below on $\overline{\spa}\left\{f_{n}\right\}_{n\ge m}$;
\item[(vi)] $\Ker \left(T-S\right)$ is dispersed, for every compact $S:F\to E$, i.e. there is no compact operator that coincides with $T$ on a non-dispersed subspace;
\item[(vii)] A restriction of $T$ to any anti-dispersed subspace is USF;
\item[(viii)] No restriction of $T$ to a non-dispersed subspace is compact;
\item[(ix)] No restriction of $T$ to a non-dispersed subspace is SS.
\end{theorem}
\begin{proof}
First, note that (v)$\Rightarrow$(iv)$\Rightarrow$(iii)$\Rightarrow$(ii) and (ix)$\Rightarrow$(viii) follow from the implications for the corresponding classes of operators, while (viii)$\Rightarrow$(vi) is obvious. (vi),(ii)$\Rightarrow$(i) are proven similarly to the implications (iii),(xi)$\Rightarrow$(ix) in  Theorem \ref{usf}. (iv)$\Rightarrow$(v) follows from (i)$\Leftrightarrow$(x) in Theorem \ref{usf}.

(vii)$\Rightarrow$(ix): From Corollary \ref{quasi} every non-dispersed subspace $H$ contains an (infinitely-dimensional) anti-dispersed sub-subspace $G$. Hence, a $\left.T\right|_{G}$ is USF, and so $\left.T\right|_{H}$ cannot be SS, according to Proposition \ref{ss}.

(viii)$\Rightarrow$(vii): If $H$ is an anti-dispersed subspace of $F$, then it has no dispersed infinitely dimensional subspaces, and so no restriction of $T$ to an infinitely dimensional subspace of $H$ is compact. Hence, $\left.T\right|_{H}$ is USF, by virtue of Theorem \ref{usf}.\medskip

(v)$\Rightarrow$(viii): Without loss of generality we may assume $\|T\|\le 1$. Assume that there is a non-dispersed subspace $H$ of $E$ such that $\left.T\right|_{H}$ is compact. Then, there is a disjoint sequence  $\left\{f_{n}\right\}_{n\in\N}\subset \So_{F}$, and a sequence  $\left\{h_{n}\right\}_{n\in\N}\subset \So_{H}$ such that $\|f_{n}-h_{n}\|\xrightarrow[n\to\8]{} 0$. Let $m\in\N$ and $\delta>0$ be such that the restriction of $T$ to $\overline{\spa}\left\{f_{n}\right\}_{n\ge m}$ is bounded from below by $\delta$, and simultaneously $\|f_{n}-h_{n}\|<\frac{\delta}{3}$, when $n\ge m$. Then, for $n,k\ge m$ we have
\begin{align*}
\|Th_{n}-Th_{k}\|&\ge \|T\left(f_{n}-f_{k}\right)\|-\|T\left(h_{n}-f_{n}\right)\|-\|T\left(h_{k}-f_{k}\right)\|\\
&\ge \delta\|f_{n}-f_{k}\|-\|h_{n}-f_{n}\|-\|h_{k}-f_{k}\|\ge \delta\|\left|f_{n}\right|+\left|f_{k}\right|\|-\frac{2\delta}{3}\ge \frac{\delta}{3}.
\end{align*}
This contradicts the fact that $T\left\{h_{n}\right\}_{n\ge m}\subset T\Bbo_{H}$ is relatively compact.\medskip

(i)$\Rightarrow$(iv): First, note that since $\rho_{T}$ is continuous, no normalized disjoint sequence in $\So_{F}$ is $\rho_{T}$-null if and only if no almost disjoint sequence in
 $\So_{F}$ is $\rho_{T}$-null. Assume that the restriction of $T$ to $H=\overline{\spa}\left\{f_{n}\right\}_{n\in\N}$ is not USF, where $\left\{f_{n}\right\}_{n\in\N}$ is disjoint. Then, from Theorem \ref{usf} there is a weakly null normalized net in $H$ that is $\rho_{T}$-null. From part (ii) of Lemma \ref{quas}, we can select an almost disjoint sequence $\left\{h_{n}\right\}_{n\in\N}\subset\partial\Bbo_{H}$, which is $\rho_{T}$-null.
\end{proof}

\begin{corollary}\label{adns}Any DNS operator maps anti-dispersed subspaces onto closed subspaces. It also maps closed bounded subsets of anti-dispersed subspaces onto closed sets.
\end{corollary}

Theorem \ref{dns} allows us to reduce consideration of an operator to considering the seminorm induced by it. Accordingly, let us call an operator $T:F\to E$ \emph{strictly DNS} if there is $\delta>0$ such that $\liminf\limits_{n\to\8}\rho_{T}\left(f_{n}\right)>\delta$, for every disjoint $\left\{f_{n}\right\}_{n\in\N}\subset\So_{F}$. Equivalently, this means that there is no disjoint $\left\{f_{n}\right\}_{n\in\N}\subset\So_{F}$ such that $\rho_{T}\left(f_{n}\right)\le\delta$. Clearly, every strictly DNS operator is DNS. Let us call $T$ \emph{weakly/strongly DNS} if $\rho_{T}$ is un-/una-complementary. Of course, these two classes coincide if $F$ is order continuous.

\begin{corollary}\label{cdns}The class of strictly/weakly/strongly DNS operators is open in $\Lo\left(F,E\right)$. Moreover, for $T\in \Lo\left(F,E\right)$ the following is true:
\item[(i)] If $T$ is (strictly/weakly/strongly) DNS, then $\Ker T$ is (strictly/weakly/strongly) dispersed. The converse is true if $T$ has a closed range.
\item[(ii)] If $G$ is a Banach space and $S\in \Lo\left(E,G\right)$ is such that $ST$ is (strictly/weakly/strongly) DNS, then so is $T$.
\item[(iii)] If additionally $S$, and $R\in \Lo\left(F\right)$ are bounded from below, then $T$ and $STR$ are (not) (strictly/weakly/strongly) DNS simultaneously.
\end{corollary}
\begin{proof}
The main statement is easy to see for the strictly DNS case, while for the weakly/strongly DNS operators it follows from part (i) of Proposition \ref{complo}.

(i): The direct implication is easy to see while the converse follows from  Proposition \ref{closed}.

(ii) and (iii) for the DNS case follow from part (i) of the Theorem \ref{dns}, for strictly DNS case they follow easily from the definition, and for weakly/strongly DNS case they are consequence of parts (ii) and (iii) of Proposition \ref{complo}
\end{proof}

The properties of DNS operators lead to the following characterization of the dispersed subspaces (equivalence of (i) and (iv) was established in \cite{gmm}):

\begin{proposition}\label{sumcl} For a subspace $H$ of $F$ the following conditions are equivalent:
\item[(i)] $H$ is dispersed;
\item[(ii)] The quotient map $Q_{H}$ is a DNS operator from $F$ onto $F\slash H$;
\item[(iii)] For every anti-dispersed subspace $G$ of $F$ we have that $G \cap H$ is finitely dimensional and $G + H$ is closed;
\item[(iv)] For every disjoint sequence $\left\{f_{n}\right\}_{n\in\N}\subset F$ we have that $\overline{\spa}\left\{f_{n}\right\}_{n\in\N} \cap H$ is finitely dimensional and $\overline{\spa}\left\{f_{n}\right\}_{n\in\N} + H$ is closed.
\end{proposition}
\begin{proof}
(i)$\Rightarrow$(ii) follows from part (i) of Corollary \ref{cdns}. (iii)$\Rightarrow$(iv) follows from the fact that the closed span of a disjoint sequence is anti-dispersed, according to Proposition \ref{anti}.

(ii)$\Rightarrow$(iii): Since every infinitely dimensional subspace of $H$ has to be dispersed, and no infinitely dimensional subspace of $G$ can be dispersed, $\dim\left(G \cap H\right)<\8$. From Corollary \ref{adns} we have that $Q_{H}G$ is closed in $F\slash H$, from where $G+H=Q_{H}^{-1}G$ is closed.

(iv)$\Rightarrow$(i): Let $\left\{f_{n}\right\}_{n\in\N}\subset \So_{F}$ be disjoint. From part (iv) of Theorem \ref{basic}, there is $m\in\N$ such that $\overline{\spa}\left\{f_{n}\right\}_{n\ge m} \cap E=\left\{0_{F}\right\}$ and $\overline{\spa}\left\{f_{n}\right\}_{n\ge m} + E$ is closed. Hence, $\overline{\spa}\left\{f_{n}\right\}_{n\ge m} + E\simeq \overline{\spa}\left\{f_{n}\right\}_{n\ge m} \oplus_{\8} E$, and so $\rho_{E}\left(f_{n}\right)$ is bounded from below.
\end{proof}

\begin{corollary}\label{fico} If $G$ and $H$ are subspaces of $F$ such that $G$ is (anti-)dispersed and is of finite co-dimension in $H$, then $H$ is also (anti-)dispersed.
\end{corollary}

It turns out that the implications between different variations of the DNS property mirror the ones between variations of disperseness.

\begin{proposition}\label{sdns} Let $T:F\to E$ be an operator. Then:
\item[(i)] If $T$ is DNS, then it is weakly DNS.
\item[(ii)] If $T$ is strongly DNS, then it is strictly DNS.
\end{proposition}
\begin{proof}
(i): Assume that there is a net in $\So_{F}$, which is null simultaneously in the un-topology and $\rho_{T}$. By a slight modification of the proof of Kadec-Pelczynski theorem \cite[Theorem 3.2]{dot}, we can select from this net an almost disjoint sequence $\left\{f_{n}\right\}_{n\in\N}\subset \So_{F}$, which is still $\rho_{T}$-null. Contradiction.\medskip

(ii): From part (v) of Theorem \ref{compl} there is a una-neighborhood $U$ of $0_{F}$ which does not intersect with the set $\left\{f\in\So_{F},~ \rho_{T}\left(f\right)<\delta\right\}$, for some $\delta>0$. Then, since every disjoint normalized sequence is una-null, its tail is contained in $U$, and so it is $\delta$-separated from $0_{F}$ with respect to $\rho_{T}$.
\end{proof}

Now we are equipped to partially answer a question from \cite{gmm} whether every DNS operator is strictly DNS (in our terminology), and consequently whether DNS operators form an open set in $\Lo\left(F,E\right)$. It turns out that this is the case in order continuous Banach lattices.

\begin{theorem} If $F$ is order continuous, then for $T\in\Lo\left(F,E\right)$ the following conditions are equivalent:
\item[(i)] $T$ is DNS;
\item[(ii)] $T$ is strictly DNS;
\item[(iii)] No un-null sequence (net) in $\So_{F}$ is $\rho_{T}$-null;
\item[(iv)] There are $\varepsilon,\delta>0$ and $h\in F_{+}$ such that $\rho_{T}\left(f\right)\ge\delta$ as soon as $f\in \So_{F}$ with $\|\left|f\right|\wedge h\|<\varepsilon$.
\end{theorem}
\begin{proof}
It is clear that (iv) implies (iii) for nets, (iii) for sequences implies (i), (ii)$\Rightarrow$(i) trivially, and (i)$\Rightarrow$(iv)$\Rightarrow$(ii) is the content of Proposition \ref{sdns}, since in order continuous Banach lattices the notions of strongly DNS and weakly DNS coincide.
\end{proof}

\begin{remark}
An alternative to reproving Kadec-Pelczynski theorem to obtain part (i) of  Proposition \ref{sdns} in order to get the implication (i)$\Rightarrow$(iv) is the following. Since $F$ is order continuous, the un-topology is Frechet-Urysohn (see \cite[Proposition 7.3]{kmt}). Hence, a DNS operator $T$ is weakly DNS, as otherwise from Theorem \ref{compl} there would be a normalized sequence which is simultaneously un-null and $\rho_{T}$-null, and so contains an almost disjoint (and still $\rho_{T}$-null) subsequence by the regular Kadec-Pelczynski theorem. This approach however is not always applicable, since the un-topology is not always Frechet-Urysohn (see \cite[Example 1.3]{kmt}).
\qed\end{remark}

\begin{question}
If $F$ is such that every dispersed subspace is strictly dispersed, does it follow that every DNS operator is strictly DNS?
\end{question}

Let us consider the counterpart to the DNS operators. An operator $S:F\to E$ is called \emph{disjointly strictly singular (DSS)} if one of the following equivalent conditions is satisfied (equivalence of (i), (ii) and (iv) was established in \cite{gmm}):

\begin{proposition}\label{dss}For $T\in\Lo\left(F,E\right)$ the following conditions are equivalent:
\item[(i)] If $\left\{f_{n}\right\}_{n\in\N}\subset F$ is disjoint, then $S$ is not bounded from below on $\overline{\spa}\left\{f_{n}\right\}_{n\in\N}$;
\item[(ii)] If $\left\{f_{n}\right\}_{n\in\N}\subset F$ is disjoint, then the restriction of $S$ to $\overline{\spa}\left\{f_{n}\right\}_{n\in\N}$ is SS;
\item[(iii)] Any restriction of $S$ to an anti-dispersed subspace is strictly singular;
\item[(iv)] No restriction of $S$ to a non-dispersed subspace is bounded from below;
\item[(v)] No restriction of $S$ to a non-dispersed subspace is USF;
\item[(vi)] Every non-dispersed subspace contains an almost disjoint $\rho_{S}$-null sequence.
\end{proposition}
\begin{proof}
(v)$\Rightarrow$(iv)$\Rightarrow$(iii)$\Rightarrow$(ii)$\Rightarrow$(i) are clear. (vi)$\Rightarrow$(v) follows from the fact that each almost disjoint sequence contains a basic one. (v)$\Rightarrow$(vi) follows from condition (xii) of Theorem \ref{usf} and part (ii) of Lemma \ref{quas}.

(i)$\Rightarrow$(v): Assume that there is a non-dispersed subspace $E$ such that $\left.S\right|_{E}$ is USF. Then from part (i) of Lemma \ref{quas} there is a disjoint sequence $\left\{f_{n}\right\}_{n\in\N}\subset F$ and an isomorphism $T$ on $F$ such that $Id_{F}-T$ is compact and $TH\subset E$, where $H=\overline{\spa} \left\{f_{n}\right\}_{n\in\N}$. Therefore, $\left.ST\right|_{H}$ is USF, and $S-ST$ is compact. Hence, $\left.S\right|_{H}$ is USF, and so from part (x) of Theorem \ref{usf} there is $m\in\N$ such that $S$ is bounded from below on $\overline{\spa} \left\{f_{n}\right\}_{n\ge m}$. Contradiction.
\end{proof}

Since the sum of SS operators is SS, the condition (ii) allows to conclude that the sum of DSS operators is DSS. Also, from the corresponding fact about SS and USF operators it follows that $T-S$ is DNS, for every DNS operator $T\in\Lo\left(F,E\right)$ and DSS operator $S\in\Lo\left(F,E\right)$, and in particular $T$ and $S$ can coincide only on dispersed subspaces. The condition (v) together with Corollary \ref{usf1} imply that the collection of DSS operators is closed. Finally, modifying condition (vi) gives the following characterization.

\begin{corollary}
If $T\in\Lo\left(F,E\right)$ is not complementary to the un-topology on any non-dispersed subspace, then it is DSS. The converse holds if $F$ is order-continuous.
\end{corollary}

\begin{remark}\label{dd}
Since a restriction of a USF / bounded from below / SS operator is of the same type, and $\left\{f_{n}\right\}_{n\in\N}\subset \spa\left\{f^{+}_{n},f^{-}_{n}\right\}_{n\in\N}$, in the conditions (iv) and (v) of Theorem \ref{dns}, and in the condition (ii) of Proposition \ref{dss} it is enough to consider only positive disjoint sequences.
\qed\end{remark}

\section{The case of positive disjoint sequences}\label{lns}

In this section we will say few words about a certain generalization of DNS and DSS operators. Again, throughout the section $F$ is a Banach lattice, whereas $E$ is a Banach space. An operator $S:F\to E$ is \emph{called lattice strictly singular (LSS)}, if no restriction of $S$ to a closed sublattice of $F$ is bounded from below. It is easy to see that this condition is equivalent to not being bounded from below on a closed span of any positive disjoint sequence. Therefore, it is clear that every DSS operator is LSS, but it is unknown whether the converse is true. Moreover, it was shown in \cite{flt} that the converse is indeed true if one shows that the sum of LSS operators is LSS. Alternatively, due to Remark \ref{dd}, it is enough to prove that any restriction of $S$ to the closed span of a positive disjoint sequence is strictly singular. Even more precisely, using the results from \cite{flt} it is enough to show that the classes of DNS and DSS operators coincide on reflexive separable discrete Banach lattices.\medskip

The counterpart to LSS operators is defined analogously to DNS operators. In order to study them, we need to slightly modify the tools from Section \ref{comple}.

\begin{proposition}\label{pco}Let $\tau$ and $\pi$ be linear topologies on an ordered normed space $H$, weaker than the norm topology. The following conditions are equivalent:
\item[(i)] There is no net in $\So_{H}\cap H_{+}$, which is null with respect to both $\tau$ and $\pi$;
\item[(ii)] There are open neighborhoods $U\in\tau$ and $V\in\pi$ of $0_{H}$ with $U\cap V\cap \So_{H}\cap H_{+}=\varnothing$;
\item[(iii)] There are open neighborhoods $U\in\tau$ and $V\in\pi$ of $0_{H}$ such that $U\cap V\cap H_{+}\subset \Bo_{H}$.\medskip

If additionally $H$ is a normed lattice and $\tau$ and $\pi$ are locally solid, then the conditions above are equivalent to
\item[(iv)] $\tau$ and $\pi$ are complementary.
\end{proposition}
\begin{proof}
(iii)$\Rightarrow$(ii)$\Rightarrow$(i) are obvious, while (i)$\Rightarrow$(ii)$\Rightarrow$(iii) are proven similarly to the proof of Theorem \ref{compl}. (iv)$\Rightarrow$(iii) is trivial, while the converse follows from taking $U$ and $V$ to be solid.
\end{proof}

\begin{proposition}\label{pco2} Let $G$ be a closed subspace of a normed lattice $H$, and let $\tau$ be a linear topology on $H$, weaker than the norm topology. Then $\rho_{G}$ and $\tau$ satisfy the conditions (i)-(iii) of Proposition \ref{pco} if and only if $0_{H}$ is $\tau$-separated from $\So_{G}\cap H_{+}$.
\end{proposition}
\begin{proof}
The necessity is easy to see, while the proof of sufficiency is similar to the proof of (i)$\Rightarrow$(iii) in Proposition \ref{compl2}, with the addition of observation that $\left|e-f\right|\le \left|e-f^{+}\right|$, in the case when $e\ge 0$.
\end{proof}

Let us call a subspace $E$ of $F$ \emph{positively dispersed} if it contains no positive almost disjoint sequences, or equivalently, there is no positive disjoint $\left\{f_{n}\right\}_{n\in\N}\subset \So_{F}$ such that $\lim\limits_{n\to\8} \rho_{E}\left(f_{n}\right)=0$. Accordingly, we will call $E$ \emph{strictly positively dispersed} if there is $\delta>0$ such that there is no positive disjoint $\left\{f_{n}\right\}_{n\in\N}\subset \So_{F}$ with $\rho_{E}\left(f_{n}\right)<\delta$.  We will call $E$ \emph{weakly/strongly positively dispersed} if $0_{F}$ is separated from $\So_{E}\cap F_{+}$ with respect to the un-/una-topology. Note that $\So_{E}\cap F_{+}$ can be empty, and so a lot of subspaces of $F$ are strongly positively dispersed just on the grounds of that. It is easy to see that strongly positively dispersed$\Rightarrow$strictly positively dispersed$\Rightarrow$positively dispersed$\Rightarrow$weakly positively dispersed.\medskip

We will say that $T\in\Lo\left(F,E\right)$ is \emph{lattice non-singular (LNS)} if there is no normalized positive disjoint $\rho_{T}$-null sequence. One can also introduce strictly/weakly/strongly LNS operators and establish the same relationships between them as in the DNS case. Obviously, every DNS operator is LNS, and one may wonder whether the converse is true.

\begin{question}
Is every LNS operator DSN?
\end{question}

In fact, from Remark \ref{dd}, this question is equivalent to whether every LNS operator on a discrete order continuous Banach lattice is USF.

\section{Acknowledgements}

The author wants to thank Vladimir Troitsky for bringing the author's attention to the topic, and also for contributing ideas to Proposition \ref{str} and Example \ref{rad}. The rest of the credit for Example \ref{rad} goes to Radomyra Shevchenko. The author also wants to thank Bill Johnson who contributed an idea for Theorem \ref{ref} and the service \href{mathoverflow.com/}{MathOverflow} which made it possible.

\end{document}